\documentclass[12pt]{amsart}

\usepackage[T1]{fontenc}
\usepackage[utf8]{inputenc} 
\usepackage[english]{babel}	

\usepackage{amsmath} 
\usepackage{amssymb} 
\usepackage{mathtools} 
\usepackage{amsthm}	
\usepackage{amsfonts} 
\usepackage{tikz-cd} 
\usepackage{eufrak} 

\usepackage{enumerate}

\usepackage{algorithmic}
\usepackage[ruled]{algorithm}
\usepackage{caption}

\usepackage{ulem}

\DeclareMathOperator{\sing}{Sing} 

\newtheorem{theorem}{Theorem}[section]

\newtheorem{corollary}[theorem]{Corollary}
\theoremstyle{definition}

\theoremstyle{remark}
\newtheorem{remark}[theorem]{Remark}

\newcommand{\dc}{\ensuremath{\mathcal{D}}}

\newcommand{\hc}{\ensuremath{\mathcal{H}}}

\newcommand{\Sc}{\ensuremath{\mathcal{S}}}

\newcommand{\bH}{\mathbb{H}}
\newcommand{\bZ}{\mathbb{Z}}

\newcommand{\bQ}{\mathbb{Q}}

\def\bin #1#2 {\left( \matrix { #1 \cr #2 \cr } \right) }

\newcommand{\tS}{\widetilde{{\ensuremath{\mathcal{S}}}}}

\begin{document}

\title[Polynomial identities related to Special Schubert varieties]
{Polynomial identities related to Special Schubert varieties}

\author{Francesca Cioffi }
\address{Universit\`a di Napoli Federico II, 
Dipartimento di Matematica e
Applicazioni \lq\lq R. Caccioppoli\rq\rq,  Via
Cintia, 80126 Napoli, Italy} \email{francesca.cioffi@unina.it}

\author{Davide Franco }
\address{Universit\`a di Napoli Federico II, 
Dipartimento di Matematica e
Applicazioni \lq\lq R. Caccioppoli\rq\rq,  Via
Cintia, 80126 Napoli, Italy} \email{davide.franco@unina.it}

\author{Carmine Sessa }
\address{Universit\`a di Napoli Federico II, 
Dipartimento di Matematica e
Applicazioni \lq\lq R. Caccioppoli\rq\rq,  Via
Cintia, 80126 Napoli, Italy} \email{carmine.sessa2@unina.it}

\abstract Let $\mathcal S$ be a single condition Schubert variety with an arbitrary number of strata. Recently, an explicit description of the summands involved in the decomposition theorem applied to such a variety has been obtained in a paper of the second author. Starting from this result, we provide an explicit description of the Poincar\'{e} polynomials of the intersection cohomology of $\mathcal S$ by means of the Poincar\'{e} polynomials of its strata, obtaining interesting polynomial identities relating Poincar\'{e} polynomials of several Grassmannians, both by a local and by a global point of view. We also present a symbolic study of a particular case of these identities.



\endabstract
\maketitle

\section{Introduction}

In the paper \cite{DGF2strata}, it was shown how one can obtain suitable polynomial identities from the study of the intersection cohomology of Schubert varieties with two strata (compare with \cite[p.~115]{DGF2strata}). The aim of our work is to extend the same approach to {\it Special Schubert varieties with an arbitrary number of strata}, by showing that the Poincar\'e polynomials of their intersection cohomology naturally leads to a class of tricky polynomial identities. In the final Appendix, we provide some of the numerical tests for the polynomial identities that we obtained in the meantime, and a symbolic study of a particular case.

The starting point of our analysis is the main result of the paper \cite{Forum}, which we now summarize. 
Let $\mathcal S$ be a {\it single condition Schubert variety or special Schubert variety} of dimension $n$
(see \cite[p. 328]{CGM} and \cite[Example 8.4.9]{Kirwan}). As  it is well known, $\Sc$ admits two standard resolutions: a
{\it small resolution} $\xi: \dc \to \Sc$ \cite[Definition 8.4.6]{Kirwan} and a (usually) non-small one
$\pi: \tS \to \Sc$ \cite[\S 3.4 and Exercise 3.4.10]{Manivel}. We will describe both resolutions $\pi $ and $\xi $ in \S 2.4.
By \cite[\S 6.2]{GMP2} and \cite[Theorem 8.4.7]{Kirwan}, we have 
\begin{equation}\label{decsmall}
IC^{\bullet}_{\Sc}\simeq R\xi {_*}\bQ_{\dc}[n] \quad\text{{\rm  in}} \quad D^b_c(\Sc),
\end{equation}
where $IC^{\bullet}_{\Sc}$
denotes the \textit{intersection cohomology complex} of $\Sc$ \cite[p. 156]{Dimca2},  and $D^b_c(\Sc)$ is the constructible derived category  
 of sheaves of $\mathbb Q$-vector spaces on $\Sc$. 

By the celebrated Decomposition theorem \cite{BBD,DeCMHodge,DeCMReview,Saito}, 
 the intersection cohomology complex of $\Sc$ is also a direct summand of 
$R\pi {_*}\bQ_{\tS}[n]$ in $D^b_c(\Sc)$. Specifically, the Decomposition theorem says that there is a decomposition in $D^b_c(\Sc)$ \cite[Theorem 1.6.1] {DeCMReview} 
\begin{equation}\label{decomposition}
R\pi {_*}\bQ_{\tS}[n]\cong \oplus_{i\in \bZ }\sideset{_{}^{p}}{_{}^{i}}{\mathop \hc} (R\pi {_*}\bQ_{\tS}[n])[-i],
\end{equation}
where
$\sideset{_{}^{p}}{_{}^{i}}{\mathop \hc}(R\pi {_*}\bQ_{\tS}[n])$ denote the {\it perverse cohomology sheaves}
\cite[\S 1.5]{DeCMReview}. Furthermore, the perverse cohomology sheaves $\sideset{_{}^{p}}{_{}^{i}}{\mathop \hc}(R\pi {_*}\bQ_{\tS}[n])$ are semisimple, i.e.~direct sum of  intersection cohomology complexes
of semisimple local systems, supported in the smooth strata of $\Sc$.  

In the paper \cite{Forum}, the summands involved in \eqref{decomposition} are explicitly described.
It turns out  that the semisimple local systems involved in the decomposition are constant sheaves supported in the smooth strata of $\pi$. In other words, the decomposition \eqref{decomposition} takes the following form
\begin{equation}\label{decimpr}
R\pi {_*}\bQ_{\tS}\cong \bigoplus_{p, q}IC_{\Delta_p}^{\bullet}[q]^{\oplus m_{pq}}
\end{equation}
for suitable multiplicities $m_{pq}\in \mathbb N_0$ (that are computed in \cite[Theorem 3.5]{Forum}) and where the strata $\Delta_p$ are special Schubert varieties, as well. 

Following the same lines as in \cite[section $4$]{DGF2strata}, our main aim is to deduce some classes of polynomial identities from the isomorphism \eqref{decimpr}. Specifically, we are going to prove a class of \textit{local identities} as well as a class of \textit{global identities}.

Our first task is accomplished in Theorem \ref{local}. In a nutshell, the argument behind our local polynomial identity rests on the remark that {\it each summand of \eqref{decimpr} is a direct sum of shifted trivial local systems in}   $D_c^b(\Delta_p^0)$, when restricted to the smooth part  $\Delta_p^0$ of each stratum $\Delta_p$. This fact follows by applying the Leray-Hirsch theorem (see \cite[Theorem 7.33]{Voisin}, \cite[Lemma 2.5]{DGF2}) to the summands, that are described  on $\Delta_p^0$ by means of  suitable Grassmann fibrations.  This implies that we are allowed to associate a Poincar\'e polynomial to each summand of \eqref{decimpr}, thus providing our local identity in the stratum $\Delta_p^0$ (for more details compare with \S 3).

As for the global polynomial identities, the idea is very similar to that of \cite[section 4]{DGF2strata} (compare with Theorem \ref{global}). From \eqref{decimpr} we deduce an isomorphism among the i-th hypercohomology spaces
\begin{equation}\label{decimprhyper1}
\bH^i(R\pi {_*}\bQ_{\tS})\cong \bigoplus_{p, q}\bH^i(IC_{\Delta_p}^{\bullet}[q]^{\oplus m_{pq}}),
\end{equation}
that leads to an equality of the corresponding Poincar\'e polynomials
\begin{equation}\label{decimprhyper2}
\sum_i t^i\dim\bH^i(R\pi {_*}\bQ_{\tS})= \sum_{i,\, p,\, q} t^i\dim\bH^i(IC_{\Delta_p}^{\bullet}[q]^{\oplus m_{pq}}).
\end{equation}
Again, all summands of \eqref{decimprhyper2} are determined by means of Leray-Hirsch theorem as Poincar\'e polynomials of suitable Grassmann fibrations.

We also observe that an \textit{explicit inductive algorithm for the computation of the Poincar\'e polynomials of the intersection cohomology of Special Schubert varieties} straightforwardly follows from our results (see Corollary \ref{Corollary} and Remark \ref{ExampleIC}). Although these Poincar\'e polynomials are already known, the availability of an algorithm for their computation could be the starting point for obtaining an analogous algorithm for all Schubert varieties in a future paper, being an explicit formula in this general case not known yet.

In the Appendix, we give an example of proof of the global polynomial identity of Theorem~\ref{global} in a particular case, by algebraic manipulation only, with a divide and conquer strategy.

\medskip
\section{Basic facts and notations}

\subsection{Preliminaries}

Throughout the paper, we shall work with $\mathbb{Q}$-coefficients \textit{cohomology groups}; that is, for any complex variety $V$ and any integer $k$, $H^{k}(V) = H^{k}(V, \mathbb{Q})$.
Let $D_{c}^{b}(V)$ denote the \textit{derived category of bounded constructible complexes of sheaves} $\mathcal{F^{\bullet}}$ on $V$ (\cite[\textsection 1.3 and \textsection 4.1]{Dimca2}, \cite[\textsection 1.5]{DeCMReview}). The symbol $\mathbb{H}^{k}(\mathcal{F}^{\bullet})$ stands for the $k$-th \textit{hypercohomology group} of $\mathcal{F}^{\bullet}$ (\cite[Defintion 2.1.4]{Dimca2}), while $IC_{V}^{\bullet}$ represents the \textit{intersection cohomology complex} of $V$ (\cite[\textsection 5.4]{Dimca2} and \cite[\textsection 1.5, \textsection 2.1]{DeCMReview}). Lastly, the intersection cohomology groups of a pure $n$-dimensional complex algebraic variety $V$ are given by (\cite[Definition 5.4.3]{Dimca2})
\begin{equation*}
IH^{k}(V) = IH^{k}(V, \mathbb{Q}) = \mathbb{H}^{k}(V, IC_{V}^{\bullet} \left[ - n \right]).
\end{equation*}

\subsection{Decomposition theorem}

The \textit{Decomposition theorem}, which was proved by A. Beilinson, J. Bernstein and P. Deligne in \cite{BBD}, is a tool of paramount importance: most of our results descend from it directly.

\begin{theorem}(Decomposition theorem, \cite[\textsection (1.6.1)]{DeCMReview})
Let $f: X \rightarrow Y$ be a proper map of complex algebraic varieties. There is an isomorphism in the constructible bounded derived category $D_{c}^{b}(Y)$
	\begin{equation*}
	Rf_{*} IC_{X} \cong \bigoplus_{i \in \mathbb{Z}} \prescript{\mathfrak{p}}{}{\mathcal{H}}^{i}(Rf_{*} IC_{X}) \left[ - i \right].
	\end{equation*}

Furthermore, the perverse sheaves $\prescript{\mathfrak{p}}{}{\mathcal{H}}^{i}(Rf_{*} IC_{X})$ are semisimple; i.e., there is a decomposition into finitely many disjoint locally closed and nonsingular subvarieties $Y = \coprod S_{\beta}$ and a canonical decomposition into a direct sum of intersection complexes of semisimple local systems
	\begin{equation*}
	\prescript{\mathfrak{p}}{}{\mathcal{H}}^{i}(Rf_{*} IC_{X}) \cong \bigoplus_{\beta} IC_{\overline{S_{\beta}}}(L_{\beta}).
	\end{equation*}
\end{theorem}

Roughly speaking, Decomposition theorem states that, under mild hypotheses, the direct image of the intersection cohomology complex of a complex algebraic variety can be thought of as the direct sum of \textit{intermediate extensions}  (\cite[\S 2.7]{DeCMReview}) of semisimple local systems (\cite[\S 2.5]{Dimca2}).

In the literature
one can find different approaches to the Decomposition Theorem
\cite{BBD}, \cite{DeCMHodge}, \cite{DeCMReview}, \cite{Saito},
{\cite{Williamson}, which is a very general result but also rather implicit.
On the other hand,  there are many special cases for which the Decomposition
Theorem admits a simplified and explicit approach. One of these is  the case of
varieties with isolated singularities \cite{N,DGF3,DGF4}. For instance, in the work  \cite{DGF3}, a simplified approach to the Decomposition Theorem for varieties with isolated singularities is developed, in connection with the existence of a \textit{natural Gysin morphism}, as defined in \cite[Definition 2.3]{DGF1} (see also \cite{RCMP} for other applications of the Decomposition Theorem to the Noether-Lefschetz Theory).

\subsection{Grassmannians and Poincaré polynomials}\label{Sec01}

We shall denote by $\mathbb{G}_{k}(\mathbb{C}^{n})$ the Grassmannian of $k$-vector subspaces of $\mathbb{C}^{n}$; that is, the set of all $k$-dimensional subspaces of $\mathbb{C}^{n}$. More in general, we can extend this definition by replacing $\mathbb{C}^{n}$ with any complex vector space $V$ (see \cite[\textsection 6]{Har1992} \cite[\textsection 1.5]{GrHa1978} \cite[\textsection 3.1]{Manivel}).

Let $X$ be a topological space. The \textit{Poincaré polynomial} $H_{X}$ of its \textit{cohomology} and the \textit{Poincaré polynomial} $IH_{X}$ of its \textit{intersection cohomology} (later on, they will be simply called Poincaré polynomials) are given by
	\begin{equation*}
	H_{X} = \sum_{\alpha \in \mathbb{Z}} \dim_{\mathbb{Q}} H^{\alpha}(X) \cdot t^{\alpha} \qquad \mbox{and} \qquad IH_{X} = \sum_{\alpha \in \mathbb{Z}} \dim_{\mathbb{Q}} IH^{\alpha}(X) \cdot t^{\alpha},
	\end{equation*}
respectively. When $X = \mathbb{G}_{k}(\mathbb{C}^{l})$, we have the following explicit formula of the Poincaré polynomial (see \cite[p. 328]{CGM}, \cite[\textsection 2 (vi), (vii), (viii)]{DGF2strata})
	\begin{equation*}
	H_{\mathbb{G}_{k}(\mathbb{C}^{l})} = \frac{P_{l}}{P_{k} P_{l - k}},
	\end{equation*}
where we assume $P_{0} = 1$ and take
	\begin{equation*}
	P_{\alpha} = h_{0} \cdot \ldots \cdot h_{\alpha - 1}, \> \forall \alpha > 0 \quad \mbox{and} \quad h_{\alpha} = \sum_{i = 0}^{\alpha} t^{2 \alpha}, \> \forall \alpha \geq 0.
	\end{equation*}

\subsection{Special Schubert Varieties}

In this subsection we collect some facts concerning special Schubert varieties and their resolutions. For more details and explanations we refer the reader to \cite[\S 2.2 - 2.6]{Forum}.

Let $i, j, k, l$ be integers such that
	\begin{equation*}
	0 < i < k \leq j < l \mbox{ and } r = k - i < l - j = c
	\end{equation*}
and fix a $j$-dimensional subspace $F \subseteq \mathbb{C}^{l}$. We are working with \textit{single condition (or special) Schubert varieties}
	\begin{equation*}
	\mathcal{S} = \{ V \in \mathbb{G}_{k}(\mathbb{C}^{l}): \dim(V \cap F) \geq i \}
	\end{equation*}
and we are considering the Whitney stratification
	\begin{equation*}
	\Delta_{1} \subset \ldots \subset \Delta_{r} \subset \Delta_{r + 1} = \mathcal{S}
	\end{equation*}
where, for any $p$,
	\begin{equation*}
	\Delta_{p} = \{ V \in \mathbb{G}_{k}(\mathbb{C}^{l}): \dim(V \cap F) \geq i_{p} = k - p + 1 \}
	\end{equation*}
is a special Schubert variety, as well, and $\Delta_{p} = \sing \Delta_{p + 1}$.
\\
For any $0 < q < p \leq r + 1$ there is a commutative diagram
	\begin{center}
		\begin{tikzcd}
		\Delta^{0}_{pq} \arrow[r, hook] \arrow[d, "\rho_{pq}", swap] & \tilde{\Delta}_{p} \arrow[d, "\pi_{p}"]\\
		\Delta^{0}_{q} \arrow[r, hook, swap, "i"] & \Delta_{p}
		\end{tikzcd}
	\end{center}
where
	\begin{align*}
	\Delta^{0}_{q} &= \Delta_{q} \backslash \sing \Delta_{q} = \{ V \in \mathbb{G}_{k}(\mathbb{C}^{l}): \dim(V \cap F) = i_{q} \},\\
	\tilde{\Delta}_{p} &= \{ (Z, V) \in \mathbb{G}_{i_{p}}(F) \times \mathbb{G}_{k}(\mathbb{C}^{l}): Z \subseteq V \},\\
	\Delta^{0}_{pq} &= \pi^{- 1}_{p}(\Delta^{0}_{q}) =\\
	&= \{ (Z, V) \in \mathbb{G}_{i_{p}}(F) \times \mathbb{G}_{k}(\mathbb{C}^{l}): Z \subseteq V \mbox{ and dim}(V \cap F) = i_{q} \},
	\end{align*}
the map
	\begin{equation*}
	\pi_{p}: (Z, V) \in \tilde{\Delta}_{p} \mapsto V \in \Delta_{p}
	\end{equation*}
is a resolution of singularities, and the function
	\begin{equation*}
	\rho_{pq}: (Z, V) \in \Delta^{0}_{pq} \mapsto V \in \Delta^{0}_{q}
	\end{equation*}
is a fibration with fibres
	\begin{equation*}
	F_{pq} = \mathbb{G}_{i_{p}}(\mathbb{C}^{i_{q}}).
	\end{equation*}

The resolutions $\pi_{p}$ are small when $k \leq c$ (see \cite[Remark 2.3]{Forum}), whereas there are other small resolutions when $k > c$ (see \cite[Proof of Lemma 3.2]{Forum}); namely
	\begin{equation*}
	\xi_{p}: (V, U) \in \mathcal{D}_{p} = \{ (V, U) \in \mathbb{G}_{k}(\mathbb{C}^{l}) \times \mathbb{G}_{k + j - i_{p}}(\mathbb{C}^{l}) : V + F \subseteq U \} \mapsto V \in \Delta_{p}.
	\end{equation*}

\medskip
\section{Local polynomial identities}

\medskip
Before we give the proof of the first theorem, we shall fix some notations in order to make it more readable. For any pair of integers $(p, q)$ with $0 < q < p$, we set
	\begin{align*}
	m_{p} &= \dim \Delta_{p} = (k + 1 - p)(j + p - k - 1) + (p - 1)(l - k),\\
	\delta_{pq} &= \dim \mathbb{G}_{p - q}(\mathbb{C}^{k - c}) = (p - q)(k - c + q - p)
	\end{align*}

and
	\begin{alignat*}{2}
	A^{\alpha}_{pq} &= H^{i}(F_{pq}), \qquad F_{pq} &&= \mathbb{G}_{i_{p}}(\mathbb{C}^{i_{q}}),\\
	D^{\alpha}_{pq} &= H^{\alpha}(T_{pq}), \qquad T_{pq} &&= \mathbb{G}_{p - q}(\mathbb{C}^{k - c}),\\
	B^{i}_{pq} &= H^{i}(G_{pq}), \qquad G_{pq} &&= \mathbb{G}_{p - q}(\mathbb{C}^{c - q + 1}).
	\end{alignat*}

\begin{theorem}\label{local}
For any pair of integers $(p, q)$ with $0 < q < p$ there is a local polynomial identity
	\begin{align*}
	\frac{P_{k - q + 1}}{P_{k - p + 1} P_{p - q}} &= \sum_{r = q + 1}^{p - 1} \left( \frac{P_{k - c}}{P_{p - r} P_{k - c - p + r}} \cdot \frac{P_{c - q + 1}}{P_{r - q} P_{c - r + 1}} \cdot t^{2d_{pr}} \right) +\\
	&+ \frac{P_{k - c}}{P_{p - q} P_{k - c - p + q}} \cdot t^{2d_{pq}} + \frac{P_{c - q + 1}}{P_{p - q} P_{c - p + 1}}.
	\end{align*}
\end{theorem}
\begin{proof}
By the Decomposition theorem \cite[Theorem 1.6.1]{DeCMReview}, we know that
	\begin{equation*}
	R \pi_{p*} \mathbb{Q}_{\tilde{\Delta}_{p}} \left[ m_{p} \right] \cong \bigoplus_{\alpha \in \mathbb{Z}} \prescript{p}{}{\mathcal{H}}^{\alpha} (R \pi_{p*} \mathbb{Q}_{\tilde{\Delta}_{p}} \left[ m_{p} \right] ) \left[ - \alpha \right].
	\end{equation*}
In \cite[Remark 3.1]{Forum} it is shown how the Leray-Hirsch theorem implies that
	\begin{equation*}
	\prescript{p}{}{\mathcal{H}}^{\alpha} (i^{*} R \pi_{p*} \mathbb{Q}_{\tilde{\Delta}_{p}})\mid _{\Delta_{q}^{0}} \cong A_{pq}^{\alpha - m_{q}} \otimes \mathbb{Q}_{\Delta_{q}^{0}} \left[ m_{q} \right],
	\end{equation*}
that is,
	\begin{equation*}
	\prescript{p}{}{\mathcal{H}}^{\alpha} (i^{*} R \pi_{p*} \mathbb{Q}_{\tilde{\Delta}_{p}} \left[ m_{p} \right] )\mid _{\Delta_{q}^{0}}
 \cong A_{pq}^{\alpha + m_{p} - m_{q}} \otimes \mathbb{Q}_{\Delta_{q}^{0}} \left[ m_{q} \right]
	\end{equation*}
(for a generalization of the Leray-Hirsch   theorem in a categorical framework we refer to \cite[Theorem 7.33]{Voisin} and \cite[Lemma 2.5]{DGF2}).
In addition, in \cite[Theorem 3.5]{Forum} it is proved that
	\begin{equation*}
	\prescript{p}{}{\mathcal{H}}^{\alpha} (R \pi_{p*} \mathbb{Q}_{\tilde{\Delta}_{p}} \left[ m_{p} \right] )\mid _{\Delta_{q}^{0}} \cong \bigoplus_{r = 0}^{p} D^{\delta_{pr} + \alpha}_{pr} \otimes R i_{pr*} IC_{\Delta_{r}}^{\bullet}\mid _{\Delta_{q}^{0}},
	\end{equation*}
where $i_{pr}: \Delta_{r} \hookrightarrow \Delta_{p}$ is an  inclusion. By \cite[Remark 3.3]{Forum}, we also have
	\begin{equation*}
	IC^{\bullet}_{\Delta_{r}}|_{\Delta_{q}^{0}} \cong \bigoplus_{\beta \in \mathbb{Z}} B_{rq}^{\beta} \otimes \mathbb{Q}_{\Delta_{q}^{0}} \left[ m_{r} - \beta \right] \cong \bigoplus_{\beta \in \mathbb{Z}} B_{rq}^{\beta + m_{r}} \otimes \mathbb{Q}_{\Delta_{q}^{0}} \left[ - \beta \right].
	\end{equation*}
Combining these results, we obtain
	\begin{align*}
	&\bigoplus_{\alpha \in \mathbb{Z}} A_{pq}^{\alpha + m_{p} - m_{q}} \otimes \mathbb{Q}_{\Delta_{q}^{0}} \left[ m_{q} - \alpha \right] \cong\\
	\cong &\bigoplus_{\alpha \in \mathbb{Z}} \left( \bigoplus_{r = q}^{p} D^{\delta_{pr} + \alpha}_{pr} \otimes \bigoplus_{\beta \in \mathbb{Z}} B_{rq}^{\beta + m_{r}} \otimes \mathbb{Q}_{\Delta_{q}^{0}} \left[ - \beta \right] \right) \left[ - \alpha \right],
	\end{align*}
where $r \in \{ q, \ldots, p \}$ because $\Delta_{r} \backslash \Delta_{q}^{0} = \emptyset$ whenever $r < q$. Since the $\gamma$-th cohomology group of a topological space is trivial when $\gamma < 0$,  we obtain
	\begin{align}\label{For01}
		\begin{split}
		&\bigoplus_{\alpha \geq - m_{p}} A_{pq}^{\alpha + m_{p}} \otimes \mathbb{Q}_{\Delta_{q}^{0}} \left[ - \alpha \right] \cong\\
		\cong &\bigoplus_{\alpha \geq - m_{p}} \left( \bigoplus_{r = q}^{p} D^{\delta_{pr} + \alpha}_{pr} \otimes \bigoplus_{\beta \geq - m_{p}} B_{rq}^{\beta + m_{r}} \otimes \mathbb{Q}_{\Delta_{q}^{0}} \left[ - \beta \right] \right) \left[ - \alpha \right].
		\end{split}
	\end{align}
The right-hand complex can be rewritten as follows
	\begin{align*}
	&\bigoplus_{\alpha \geq - m_{p}} \left( \bigoplus_{r = q}^{p} D^{\delta_{pr} + \alpha}_{pr} \otimes \bigoplus_{\beta \geq - m_{p}} B_{rq}^{\beta + m_{r}} \otimes \mathbb{Q}_{\Delta_{q}^{0}} \left[ - \beta \right] \right) \left[ - \alpha \right] \cong\\
	\cong &\bigoplus_{\alpha, \beta \geq - m_{p}} \left( \bigoplus_{r = q}^{p} D^{\delta_{pr} + \alpha}_{pr} \otimes B_{rq}^{\beta + m_{r}} \otimes \mathbb{Q}_{\Delta_{q}^{0}} \left[ - \alpha - \beta \right] \right)
	\end{align*}
and, for any $\gamma $, its $\gamma$-th term is
	\begin{equation*}
	\bigoplus_{\alpha + \beta = \gamma} \left( \bigoplus_{r = q}^{p} D^{\delta_{pr} + \alpha}_{pr} \otimes B_{rq}^{\beta + m_{r}} \otimes \mathbb{Q}_{\Delta_{q}^{0}} \right).
	\end{equation*}  

The isomorphism \eqref{For01} implies that the $\gamma$-th terms of those complexes are isomorphic for any $\gamma \geq - m_{p}$; i.e.
	\begin{equation*}
	A_{pq}^{\gamma + m_{p}} \otimes \mathbb{Q}_{\Delta_{q}^{0}} \cong \bigoplus_{\alpha + \beta = \gamma} \left( \bigoplus_{r = q}^{p} D^{\delta_{pr} + \alpha}_{pr} \otimes B_{rq}^{\beta + m_{r}} \otimes \mathbb{Q}_{\Delta_{q}^{0}} \right).
	\end{equation*}
We shall observe one last thing before we compute Poincaré polynomials. For any $n \in \mathbb{N}$, $\mathbb{G}_{0}(\mathbb{C}^{n}) = \{ 0 \}$, as it is the space of $0$-dimensional subspaces through the origin. As a consequence,
	\begin{equation*}
	H^{k}(\mathbb{G}_{0}(\mathbb{C}^{n})) \cong \begin{cases*} \mathbb{Q} &\mbox{if $k = 0$}\\ 0 &\mbox{otherwhise.} \end{cases*}
	\end{equation*}
Therefore, when $r = p$,
	\begin{equation*}
	\delta_{pp} = 0 \quad \mbox{ and } \quad D^{\alpha}_{pp} \cong \begin{cases*} \mathbb{Q} &\mbox{if $\alpha = 0$}\\ 0 &\mbox{otherwhise} \end{cases*}
	\end{equation*}
and, consequently,
	\begin{equation*}
	\bigoplus_{\alpha + \beta = \gamma} D^{\alpha}_{pp} \otimes B_{pq}^{\beta + m_{p}} \otimes \mathbb{Q}_{\Delta_{q}^{0}} \cong B_{pq}^{\gamma + m_{p}} \otimes \mathbb{Q}_{\Delta_{q}^{0}}.
	\end{equation*}
Similarly, when $r = q$, we have
	\begin{equation*}
	B^{\beta + m_{q}}_{pq} \cong \begin{cases*} \mathbb{Q} &\mbox{if $\beta = - m_{q}$}\\ 0 &\mbox{otherwhise} \end{cases*}
	\end{equation*}
and
	\begin{equation*}
	\bigoplus_{\alpha + \beta = \gamma} D^{\delta_{pq}+\alpha}_{pq} \otimes B_{qq}^{\beta + m_{q}} \otimes \mathbb{Q}_{\Delta_{q}^{0}} \cong D_{pq}^{\delta_{pq} + m_{q} + \gamma} \otimes \mathbb{Q}_{\Delta_{q}^{0}}.
	\end{equation*}
In conclusion, for any $\gamma \geq - m_{p}$, we have
	\begin{align*}
	A_{pq}^{\gamma + m_{p}} \otimes \mathbb{Q}_{\Delta_{q}^{0}} \cong &\bigoplus_{\alpha + \beta = \gamma} \left( \bigoplus_{r = q - 1}^{p - 1} D^{\delta_{pr} + \alpha}_{pr} \otimes B_{rq}^{\beta + m_{r}} \otimes \mathbb{Q}_{\Delta_{q}^{0}} \right) \oplus\\
	&\oplus \left( D_{pq}^{\delta_{pq} + m_{q} + \gamma} \otimes \mathbb{Q}_{\Delta_{q}^{0}} \right) \oplus \left( B_{pq}^{\gamma + m_{p}} \otimes \mathbb{Q}_{\Delta_{q}^{0}} \right).
	\end{align*}

Let $s = \gamma + m_{p}$ and recall that $m_{p} - m_{r} - \delta_{pr} = 2 d_{pr}$ (see \cite[\textsection 2.6]{Forum}).
	\begin{align*}
	&\bigoplus_{\alpha + \beta = \gamma} \left( \bigoplus_{r = q - 1}^{p - 1} D^{\delta_{pr} + \alpha}_{pr} \otimes B_{rq}^{\beta + m_{r}} \otimes \mathbb{Q}_{\Delta_{q}^{0}} \right) \cong\\
	\cong &\bigoplus_{\alpha' + \beta = s} \left( \bigoplus_{r = q - 1}^{p - 1} D^{\delta_{pr} + \alpha' - m_{p}}_{pr} \otimes B_{rq}^{\beta + m_{r}} \otimes \mathbb{Q}_{\Delta_{q}^{0}} \right) \cong\\
	\cong &\bigoplus_{\alpha'' + \beta' = s} \left( \bigoplus_{r = q - 1}^{p - 1} D^{\delta_{pr} + \alpha'' - m_{p} + m_{r}}_{pr} \otimes B_{rq}^{\beta'} \otimes \mathbb{Q}_{\Delta_{q}^{0}} \right) \cong\\
	\cong &\bigoplus_{\alpha'' + \beta' = s} \left( \bigoplus_{r = q - 1}^{p - 1} D^{\alpha'' - 2d_{pr}}_{pr} \otimes B_{rq}^{\beta'} \otimes \mathbb{Q}_{\Delta_{q}^{0}} \right),
	\end{align*}
where we set $\alpha' = \alpha + m_{p}$ and $\alpha'' = \alpha' - m_{r}$, $\beta' = \beta + m_{r}$. Hence, we have
	\begin{align*}
	A_{pq}^{s} \otimes \mathbb{Q}_{\Delta_{q}^{0}} \cong &\bigoplus_{\alpha + \beta = s} \left( \bigoplus_{r = q - 1}^{p - 1} D^{\alpha - 2d_{pr}}_{pr} \otimes B_{rq}^{\beta} \otimes \mathbb{Q}_{\Delta_{q}^{0}} \right) \oplus\\
	&\oplus \left( D_{pq}^{s - 2d_{pq}} \otimes \mathbb{Q}_{\Delta_{q}^{0}} \right) \oplus \left( B_{pq}^{s} \otimes \mathbb{Q}_{\Delta_{q}^{0}} \right)
	\end{align*}
for any $s \geq 0$. At long last, if we denote by
	\begin{equation*}
	a_{pq}^{s} = \dim A_{pq}^{s}, \qquad d_{pq}^{s} = \dim D_{pq}^{s}, \qquad b_{pq}^{s} = \dim B_{pq}^{s},
	\end{equation*}
we obtain identities
	\begin{equation*}
	a_{pq}^{s} = \sum_{r = q - 1}^{p - 1} \left( \sum_{\alpha + \beta = s} d_{pr}^{\alpha - 2d_{pr}} \cdot b_{rq}^{\beta} \right) + d_{pq}^{s - 2d_{pq}} + b_{pq}^{s}.
	\end{equation*}
If we formally multiply both sides by $t^{s}$,
	\begin{align*}
	a_{pq}^{s} \cdot t^{s} &= \sum_{r = q - 1}^{p - 1} \left( \sum_{\alpha + \beta = s} d_{pr}^{\alpha - 2d_{pr}} \cdot b_{rq}^{\beta} \right) \cdot t^{s} + d_{pq}^{s - 2d_{pq}} \cdot t^{s} + b_{pq}^{s} \cdot t^{s} =\\
	&= \sum_{r = q - 1}^{p - 1} \left( \sum_{\alpha + \beta = s} (d_{pr}^{\alpha - 2d_{pr}} \cdot t^{\alpha - 2d_{pr}}) \cdot (b_{rq}^{\beta} \cdot t^{\beta}) \right) \cdot t^{2d_{pr}} +\\
	&+ (d_{pq}^{s - 2d_{pq}} \cdot t^{s - 2d_{pq}}) \cdot t^{2d_{pq}} + b_{pq}^{s} \cdot t^{s}
	\end{align*}
and if we take the sum over $s \in \mathbb{Z} \mbox{ }$
	\begin{align*}
	\sum_{s} a_{pq}^{s} \cdot t^{s} &= \sum_{r = q - 1}^{p - 1} \left( \sum_{\alpha} (d_{pr}^{\alpha - 2d_{pr}} \cdot t^{\alpha - 2d_{pr}}) \cdot \sum_{\beta} (b_{rq}^{\beta} \cdot t^{\beta}) \right) \cdot t^{2d_{pr}} +\\
	&+ \sum_{s} (d_{pq}^{s - 2d_{pq}} \cdot t^{s - 2d_{pq}}) \cdot t^{2d_{pq}} + \sum_{s} b_{pq}^{s} \cdot t^{s}.
	\end{align*}
We are done, because, by definition of Poincaré polynomials, the above equality becomes
	\begin{equation*}
	H_{F_{pq}} = \sum_{r = q - 1}^{p - 1} \left( H_{T_{pr}} \cdot H_{G_{rq}} \right) \cdot t^{2d_{pr}} + H_{T_{pq}} \cdot t^{2d_{pq}} + H_{G_{pq}};
	\end{equation*}
	that is,
	\begin{align*}
	\frac{P_{k - q + 1}}{P_{k - p + 1} P_{p - q}} &= \sum_{r = q - 1}^{p - 1} \left( \frac{P_{k - c}}{P_{p - r} P_{k - c - p + r}} \cdot \frac{P_{c - q + 1}}{P_{r - q} P_{c - r + 1}} \cdot t^{2d_{pr}} \right) +\\
	&+ \frac{P_{k - c}}{P_{p - q} P_{k - c - p + q}} \cdot t^{2d_{pq}} + \frac{P_{c - q + 1}}{P_{p - q} P_{c - p + 1}}.
	\end{align*}
\end{proof}

\section{Global  polynomial identities}

We shall begin by introducing some further notations:
	\begin{align*}
	H_{p} &= \sum_{\alpha \in \mathbb{Z}} \dim_{\mathbb{Q}} H^{\alpha} (\tilde{\Delta}_{p}) \cdot t^{\alpha};\\
	I_{p} &= \sum_{\alpha \in \mathbb{Z}} \dim_{\mathbb{Q}} IH^{\alpha} (\Delta_{p}) \cdot t^{\alpha} = \sum_{\alpha \in \mathbb{Z}} \dim_{\mathbb{Q}} \mathbb{H}^{\alpha} \left( IC_{\Delta_{p}}^{\bullet} \left[ - m_{p} \right] \right) \cdot t^{\alpha};\\
	f_{pq} &= \sum_{\alpha \in \mathbb{Z}} d_{pq}^{\alpha} \cdot t^{\alpha} = \sum_{\alpha \in \mathbb{Z}} H^{\alpha}(T_{pq}) \cdot t^{\alpha}.
	\end{align*}
Recall that we  defined a small resolution $\xi_p: \mathcal{D}_{p}\to \Delta_{p}$ as 
	$$
	\xi_{p}: (V, U) \in \mathcal{D}_{p} = \{ (V, U) \in \mathbb{G}_{k}(\mathbb{C}^{l}) \times \mathbb{G}_{k + j - i_{p}}(\mathbb{C}^{l}) : V + F \subseteq U \} \mapsto V \in \Delta_{p}.
	$$
	
\begin{remark}\label{rmksmall}
The map 
\begin{equation*}
	\varphi:  U\in G:=\{U \in \mathbb{G}_{k + j - i_{p}}(\mathbb{C}^{l})\mid \,\, F\subseteq U\} \mapsto U/F \in \mathbb{G}_{k - i_{p}}(\mathbb{C}^{l - j}),
	\end{equation*}	
provides an isomorphism between $G$ and 
$\mathbb{G}_{k - i_{p}}(\mathbb{C}^{l - j})$. Therefore, we recognize $\mathcal{D}_{p}$ as the \textit{Grassmannian of subspaces of dimension $k$ of the  restriction of the tautological bundle $\mathcal{S}$ over $\mathbb{G}_{k + j - i_{p}}(\mathbb{C}^{l})$ to the subspace} $G$:
	\begin{equation*}
	\mathcal{D}_{p}\cong \mathbb{G}_k(\mathcal{S}\mid_G).
	\end{equation*}
By the   Leray-Hirsch theorem \cite[Theorem 7.33]{Voisin}, we have
\begin{align*}
	H^{\alpha}(\mathcal{D}_{p}) &\cong H^{\alpha} \left(\mathbb{G}_{k - i_{p}}(\mathbb{C}^{l - j}) \times \mathbb{G}_{k}(\mathbb{C}^{k + j - i_{p}}) \right) \cong\\
	&\cong H^{\alpha}(\mathbb{G}_{k - i_{p}}(\mathbb{C}^{l - j})) \otimes H^{\alpha}(\mathbb{G}_{k}(\mathbb{C}^{k + j - i_{p}})),
\end{align*}
(compare also with \cite[\textsection 2]{DGF2} for a discussion of the Leray-Hirsch theorem in a context which is closely related with that considered here).
\end{remark}

\medskip
The following formula, which represents the main result if this paper,  provides a strong generalization of \cite[\textsection 2, Remark 4.2]{DGF2strata}.

\begin{theorem}\label{global}
For any $q < p$, we have: 
	\begin{align*}
	\frac{P_{j} P_{l - i}}{P_{i} P_{j - i} P_{k - i} P_{l - k}} = &\frac{P_{l - j} P_{k + j - i}}{P_{k - i} P_{l + i - j - k} P_{k} P_{j - i}} +\\
	+ \sum_{s = 1}^{\min \{ k - i, k - c \} } &\frac{P_{k - c} P_{l - j} P_{k + j - i - s}}{P_{s} P_{k - c - s} P_{k - i - s} P_{l + i - j - k + s} P_{k} P_{j - i - s}} t^{2s(c - r + s)}.
	\end{align*}
\end{theorem}
\begin{proof}
The first part of the proof is similar to that of Theorem \ref{local}. We combine the Decomposition theorem \cite[Theorem 1.6.1]{DeCMReview} and \cite[Theorem 3.5]{Forum} so as to obtain
\begin{equation*}
R \pi_{p*} \mathbb{Q}_{\tilde{\Delta}_{p}} \left[ m_{p} \right] \cong \bigoplus_{\alpha \in \mathbb{Z}} \left( \bigoplus_{q = 0}^{p} D_{pq}^{\delta_{pq} + \alpha} \otimes R i_{pq*} IC^{\bullet}_{\Delta_{q}} \right) \left[ - \alpha \right];
\end{equation*}

that is,
\begin{align*}
R \pi_{p*} \mathbb{Q}_{\tilde{\Delta}_{p}} \cong &\bigoplus_{\alpha \in \mathbb{Z}} \left( \bigoplus_{q = 1}^{p - 1} D_{pq}^{\delta_{pq} + \alpha} \otimes R i_{pq*} IC^{\bullet}_{\Delta_{q}} \left[ - m_{p} - \alpha \right] \right) \oplus\\
\oplus &\bigoplus_{\alpha \in \mathbb{Z}} \left( D_{p0}^{\delta_{p0} + \alpha} \otimes R i_{p0*} IC^{\bullet}_{\Delta_{0}} \left[ - m_{p} - \alpha \right] \right) \oplus\\
\oplus &\bigoplus_{\alpha \in \mathbb{Z}} \left( D_{pp}^{\delta_{pp} + \alpha} \otimes R i_{pp*} IC^{\bullet}_{\Delta_{p}} \left[ - m_{p} - \alpha \right] \right).
\end{align*}

We have already met the term $D_{pp}^{\delta_{pp} + \alpha}$ and the second summand of the right-hand side is the zero complex since $\Delta_{0} = \emptyset$. Hence, we have
\begin{align*}
R \pi_{p*} \mathbb{Q}_{\tilde{\Delta}_{p}} &\cong \bigoplus_{\alpha \in \mathbb{Z}} \left( \bigoplus_{q = 1}^{p - 1} D_{pq}^{\delta_{pq} + \alpha} \otimes R i_{pq*} IC^{\bullet}_{\Delta_{q}} \left[ - m_{p} - \alpha \right] \right)	\oplus IC^{\bullet}_{\Delta_{p}} \left[ - m_{p} \right] =\\
&= \bigoplus_{\alpha \in \mathbb{Z}} \left( \bigoplus_{q = 1}^{p - 1} D_{pq}^{\alpha - 2d_{pq}} \otimes R i_{pq*} IC^{\bullet}_{\Delta_{q}} \left[ - m_{q} - \alpha \right] \right)	\oplus IC^{\bullet}_{\Delta_{p}} \left[ - m_{p} \right],
\end{align*}
where we have also taken  account of $m_{p} - m_{r} - \delta_{pr} = 2 d_{pr}$ (see \cite[\textsection 2.6]{Forum}).

If we apply hypercohomology, we obtain (for any $\beta \in \mathbb{Z}$)
\begin{align*}
\mathbb{H}^{\beta} (R \pi_{p*} \mathbb{Q}_{\tilde{\Delta}_{p}}) &\cong \bigoplus_{\alpha \in \mathbb{Z}} \left( \bigoplus_{q = 1}^{p - 1} D_{pq}^{\alpha - 2d_{pq}} \otimes \mathbb{H}^{\beta - \alpha} \left( R i_{pq*} IC^{\bullet}_{\Delta_{q}} \left[ - m_{q} \right] \right) \right) \oplus\\
&\oplus \mathbb{H}^{\beta} \left( IC^{\bullet}_{\Delta_{p}} \left[ - m_{p} \right] \right)
\end{align*}
By \cite[Definition 5.4.3]{Dimca2},
\begin{equation*}
IH^{\beta}(\Delta_{p}) = \mathbb{H}^{\beta} \left( IC^{\bullet}_{\Delta_{p}} \left[ - m_{p} \right] \right)
\end{equation*}
and, by \cite[Definition 2.1.4]{Dimca2} and \cite[Chapter II, (4.5)]{Iversen},
\begin{align*}
\mathbb{H}^{\beta} (R i_{pq*} IC^{\bullet}_{\Delta_{q}}) &= H^{\beta}(\Gamma(\Delta_{p}, i_{pq*} I^{\bullet})) = H^{\beta}(\Gamma(i^{- 1}_{pq*}(\Delta_{p}), I^{\bullet})) =\\
&= H^{\beta}(\Gamma(\Delta_{q}, I^{\bullet})) = \mathbb{H}^{\beta} (IC^{\bullet}_{\Delta_{q}}),
\end{align*}
where $IC^{\bullet}_{\Delta_{q}} \rightarrow I^{\bullet}$ is an injective resolution of $IC^{\bullet}_{\Delta_{q}}$ and $\Gamma$ is the global section functor. Similarly,
\begin{align*}
\mathbb{H}^{\beta}(R \pi_{p*} \mathbb{Q}_{\tilde{\Delta}_{p}}) &= H^{\beta}(\Gamma(\Delta_{p}, R \pi_{p*} \mathbb{Q}_{\tilde{\Delta}_{p}})) = H^{\beta}(\Gamma(\Delta_{p}, \pi_{p*} I^{\bullet})) =\\
&= H^{\beta}(\Gamma(\pi_{p*}^{-1}(\Delta_{p}), I^{\bullet})) = H^{\beta}(\Gamma(\tilde{\Delta}_{p}, I^{\bullet})) =\\
&= H^{\beta}(\tilde{\Delta}_{p}, \mathbb{Q}_{\tilde{\Delta}_{p}}) = H^{\beta}(\tilde{\Delta}_{p}),
\end{align*}
where $\mathbb{Q}_{\tilde{\Delta}_{p}} \rightarrow I^{\bullet}$ is an injective resolution and the last equality is \cite[Theorem 7.12, p. 242]{Iversen}.
\\
Substituting in the above isomorphism, we obtain
\begin{equation*}
H^{\beta}(\tilde{\Delta}_{p}) \cong \bigoplus_{\alpha \in \mathbb{Z}} \left( \bigoplus_{q = 1}^{p - 1} D_{pq}^{\alpha - 2d_{pq}} \otimes IH^{\beta - \alpha} (\Delta_{q}) \right) \oplus \mathbb{H}^{\beta} \left( IC^{\bullet}_{\Delta_{p}} \left[ - m_{p} \right] \right).
\end{equation*}
As we did in the proof of Theorem \ref{local}, we conclude
\begin{align*}
\sum_{\beta \in \mathbb{Z}} \dim H^{\beta}(\tilde{\Delta}_{p}) \cdot t^{\beta} &= \sum_{q = 1}^{p - 1} \sum_{\alpha, \beta \in \mathbb{Z}} d_{pq}^{\alpha - 2d_{pq}} \cdot t^{\alpha - 2d_{pq}} \cdot\\
&\cdot \dim IH^{\beta - \alpha} (\Delta_{q}) \cdot t^{\beta - \alpha} \cdot t^{2d_{pq}} + \sum_{\beta \in \mathbb{Z}} \dim IH^{\beta} (\Delta_{p}) \cdot t^{\beta},
\end{align*}
which can be compactly rewritten as
\begin{equation}\label{For02}
H_{p} = I_{p} + \sum_{q = 1}^{p - 1} f_{pq} \cdot I_{q} \cdot t^{2d_{pq}}.
\end{equation}

Again, Leray-Hirsch theorem implies that $\tilde{\Delta}_{p}$ has the same Poincar\'e polynomial as $\mathbb{G}_{i_{p}}(F) \times \mathbb{G}_{k - i_{p}}(\mathbb{C}^{l - i_{p}})$. Thus, the left-hand side is
\begin{equation*}
H_{p}= H_{\mathbb{G}_{i_{p}}(F) \times \mathbb{G}_{k - i_{p}}(\mathbb{C}^{l - i_{p}})} = H_{\mathbb{G}_{i_{p}}(F)} \cdot H_{\mathbb{G}_{k - i_{p}}(\mathbb{C}^{l - i_{p}})}.
\end{equation*}
By virtue of \cite[Formula (19)]{Forum} and Remark \ref{rmksmall}, we have 
\begin{align*}
IH^{\alpha}(\Delta_{p}) &= \mathbb{H}^{\alpha} \left( \Delta_{p}, IC^{\bullet}_{\Delta_{p}} \left[ - m_{p} \right] \right) = \mathbb{H}^{\alpha}(\Delta_{p}, R \xi_{p*} \mathbb{Q}_{\mathcal{D}_{p}}) =\\
&= H^{\alpha}(\mathcal{D}_{p}) \cong H^{\alpha}(\mathbb{G}_{k - i_{p}}(\mathbb{C}^{l - j})) \otimes H^{\alpha}(\mathbb{G}_{k}(\mathbb{C}^{k + j - i_{p}}));
\end{align*}
in other words,
\begin{equation*}
I_{p} = H_{\mathbb{G}_{k - i_{p}}(\mathbb{C}^{l - j}) \times \mathbb{G}_{k}(\mathbb{C}^{k + j - i_{p}})} = H_{\mathbb{G}_{k - i_{p}}(\mathbb{C}^{l - j})} \cdot H_{\mathbb{G}_{k}(\mathbb{C}^{k + j - i_{p}})}.
\end{equation*}

Adopting the same notations as \textsection \ref{Sec01}, we have
\begin{align*}
H_{p} &= \frac{P_{j}}{P_{i_{p}} P_{j - i_{p}}} \cdot \frac{P_{l - i_{p}}}{P_{k - i_{p}} P_{l - k}},\\
I_{p} &= \frac{P_{l - j}}{P_{k - i_{p}} P_{l - j - k + i_{p}}} \cdot \frac{P_{k + j - i_{p}}}{P_{k} P_{j - i_{p}}}
\end{align*}
and
\begin{equation*}
f_{pq}= \frac{P_{k - c}}{P_{p - q} P_{k - c - (p - q)}}.
\end{equation*}
Formula (\ref{For02}) becomes
\begin{align*}
\frac{P_{j} P_{l - i_{p}}}{P_{i_{p}} P_{j - i_{p}} P_{k - i_{p}} P_{l - k}} = &\frac{P_{l - j} P_{k + j - i_{p}}}{P_{k - i_{p}} P_{l - j - k + i_{p}} P_{k} P_{j - i_{p}}} +\\
+ \sum_{q = 1}^{\min \{ p - 1, k - c - p \} } &\frac{P_{k - c} P_{l - j} P_{k + j - i_{q}}}{P_{p - q} P_{k - c - (p - q)} P_{k - i_{q}} P_{l - j - k + i_{q}} P_{k} P_{j - i_{q}}} t^{2d_{pq}}.
\end{align*}

Since we are interested in studying the Poincaré polynomials of the Schubert variety $\mathcal{S}$, we are going to take $p = r + 1$. Bearing in mind that $i_{q} = k - q + 1$ (in particular, $i_{p} = i_{r + 1} = k - r = i$) and $c = l - j$, if we set $s = p - q = r + 1 - q$, we have (from left to right, numerators first)
\begin{align*}
l - i_{p} &= l - i,\\
k + j - i_{p} &= k + j - i,\\
k + j - i_{q} &= j + q - 1 = j + r - s = j + k - i - s,\\
j - i_{p} &= j - i,\\
k - i_{p} &= k - i,\\
l - j - k + i_{p} &= l + i - j - k,\\
k - i_{q} &= q - 1 = r - s = k - i - s,\\
l - j - k + i_{q} &= l - j - q + 1 = l - j + s - r = l - j + s - k + i\\
j - i_{q} &= j - k + q - 1 = j - k + r - s = j - i - s
\end{align*}

and the previous equality becomes ($2d_{pq} = 2(p - q)(c + 1 - q) = 2s(c + s - r)$)
\begin{align*}
\frac{P_{j} P_{l - i}}{P_{i} P_{j - i} P_{k - i} P_{l - k}} = &\frac{P_{l - j} P_{k + j - i}}{P_{k - i} P_{l + i - j - k} P_{k} P_{j - i}} +\\
+ \sum_{s = 1}^{\min \{ k - i, k - c \} } &\frac{P_{k - c} P_{l - j} P_{k + j - i - s}}{P_{s} P_{k - c - s} P_{k - i - s} P_{l + i - j - k + s} P_{k} P_{j - i - s}} t^{2s(c - r + s)}.
\end{align*}
\end{proof}

\begin{corollary}
\label{Corollary} For all $p=2,\dots,r+1$ one has:
$$
I_p=H_p-\sum_{q=1}^{p-1}P_{pq}(t)=H_p-\sum_{q=1}^{p-1}t^{2d_{pq}}f_{pq}I_q.
$$
\end{corollary}

\begin{remark}
\label{ExampleIC} From the previous
corollary we get an \textit{explicit inductive algorithm for the  computation of  Poincar\'e polynomials of the intersection cohomology of Special Schubert varieties}, which is described by the following equality, where we put $g_{pq}=t^{2d_{pq}}f_{pq}$ in order to simplify the notation:
$$
\left[\begin{matrix}
I_{r+1}\\I_r\\.\\.\\.\\I_2\\I_1\end{matrix}\right]=
\left[\begin{matrix}
1&g_{r+1,r}&g_{r+1,r-1}&g_{r+1,r-2}&\dots&g_{r+1,1}\\0&1&g_{r,r-1}&g_{r,r-2}&\dots&g_{r,1}\\
0&0&1&g_{r-1,r-2}&\dots&g_{r-1,1}\\.&.&.&.&.&.\\.&.&.&.&.&.\\0&0&0&\dots&1&g_{21}\\0&0&0&0&\dots&1
\end{matrix}\right]^{-1}
\cdot \left[\begin{matrix}
H_{r+1}\\H_r\\.\\.\\.\\H_2\\H_1\end{matrix}\right]=
$$
$$
=\sum _{k=0}^r (-1)^k
\left[\begin{matrix}
0&g_{r+1,r}&g_{r+1,r-1}&g_{r+1,r-2}&\dots&g_{r+1,1}\\0&0&g_{r,r-1}&g_{r,r-2}&\dots&g_{r,1}\\
0&0&0&g_{r-1,r-2}&\dots&g_{r-1,1}\\.&.&.&.&.&.\\.&.&.&.&.&.\\0&0&0&\dots&0&g_{21}\\0&0&0&0&\dots&0
\end{matrix}\right]^{k}
\cdot \left[\begin{matrix}
H_{r+1}\\H_r\\.\\.\\.\\H_2\\H_1\end{matrix}\right].
$$
\end{remark}

\medskip 
\begin{small}
\section{Appendix: a symbolic point of view}

In this Appendix, we consider the global polynomial identity of Theorem~\ref{global} from a symbolic point of view. Recall that the requests over the integers $i,j,k,j,l$ and the values $r:=k-i,c:=l-j$ are
$$0<i<k\leq j<l \text{ and } 0<r<c<k,$$
or, equivalently,
\begin{equation}\label{eq:geom-cond}
0<i<j \text{ and } 0<r<c < r+i \leq j.
\end{equation}




Nevertheless, if we also assume $P_{\alpha} = 0$ for every $\alpha < 0$, the polynomial identity of Theorem~\ref{global} symbolically makes sense, i.e.~the denominators do not vanish, under the following weaker assumptions:
$$0\leq i \leq k \leq j \text{ and } 0\leq r\leq c\leq k$$
or, equivalently,
\begin{equation}\label{eq:symb-cond}
0\leq i \leq j \text{ and }  0\leq r \leq c \leq r+i \leq j.
\end{equation}
However, in the few further cases $r=0$ or $c=r+i$ or $i=0$ or $i=j$, which we have under the  assumptions \eqref{eq:symb-cond}, the polynomial identity trivially holds. For the remaining case $c=r$, we obtain some experimental evidences verifying the polynomial identity for the $4$-uples $(i,j,c,r)$, where $c=r$ varies in $[2,20]$, $i$ in $[1,20]$ and $j$ in $[r+i,40]$, by direct computations performed in CoCoA5 (see \cite{CoCoA5}).

In the following, we consider the assumptions \eqref{eq:geom-cond} (except for some special cases that will be highlighted), and give a proof of that identity by a mere algebraic manipulation when $2=\min\{k-i, k-c\}$, which is the first case with a significant geometric meaning in the context of this paper.

By direct computations, we also verified that the polynomial identity of Theorem 4.2 holds for all the $4$-uples $(i,j,c,r)$ where $i$ varies in $[1,20]$, $r$ in $[2,20]$, $j$ in $[r+i,40]$ and $c$ in $[r+1,r+i-1]$. At http://wpage.unina.it/cioffifr/PolynomialIdentity, some CoCoA5 functions that perform such computations are available.

\subsection{Case $2=k-i \leq k-c$}
With the same notation of Theorem~\ref{global} and recalling that $k=i+2$ and $\ell=j+c$, the global polynomial identity becomes:
\begin{equation}\label{caso k-i=2}
\begin{array}{c}
\displaystyle{\frac{P_jP_{j+c-i}}{P_iP_{j-i}P_2P_{j+c-i-2}} = \frac{P_cP_{j+2}}{P_2P_{c-2}P_{i+2}P_{j-i}}  +} \\
\displaystyle{+t^{2(c-1)} \frac{P_{i-c+2}P_cP_{j+1}}{P_{i-c+1}P_{c-1}P_{i+2}P_{j-i-1}} + t^{4c} \frac{P_{i-c+2}P_j}{P_2P_{i-c}P_{i+2}P_{j-i-2}}}
\end{array}
\end{equation}
where only the parameters $i,j,c$ appear. Note that formula \eqref{caso k-i=2} does make sense for every $j\geq i+2$ and $c\geq 2$.
Let
\vskip 2mm
$\displaystyle{ H := \frac{P_jP_{j+c-i}}{P_iP_{j-i}P_2P_{j+c-i-2}}}$, \quad $\displaystyle{ F:= \frac{P_cP_{j+2}}{P_2P_{c-2}P_{i+2}P_{j-i}}}$, 
\vskip 2mm
$\displaystyle{ F_1 :=  t^{2(c-1)} \frac{P_{i-c+2}P_cP_{j+1}}{P_{i-c+1}P_{c-1}P_{i+2}P_{j-i-1}}}$, \quad $\displaystyle{ F_2 :=  t^{4c} \frac{P_{i-c+2}P_j}{P_2P_{i-c}P_{i+2}P_{j-i-2}}}$.  

\noindent Applying the fact that $\displaystyle{\frac{P_\alpha}{P_{\alpha-h}} =h_{\alpha-1}\dots h_{\alpha-h}}$, for every $\alpha> h\geq 0$, we compute
$\displaystyle{ F:= \frac{h_{c-2}h_{c-1}P_{j+2}}{h_1 P_{i+2}P_{j-i}}}$ and observe that

\vskip 2mm
$\displaystyle{ H = F \left[\frac{h_{j+c-i-2}h_{j+c-i-1}h_ih_{i+1}}{h_jh_{j+1}h_{c-2}h_{c-1}} \right] }$,
$\displaystyle{ F_1 =  F \left[  t^{2(c-1)} \frac{h_1h_{i-c+1}h_{j-i-1}}{ h_{j+1}h_{c-2} }\right] }$,
\vskip 2mm
$\displaystyle{ F_2 = F  \left[t^{4c}  \frac{h_{i-c}h_{i-c+1}h_{j-i-2}h_{j-i-1}}{h_{c-2}h_{c-1}h_jh_{j+1}} \right]  }$.

\noindent Hence, letting 
\begin{align*}
F(i,j,c)&:= \frac{h_{j+c-i-2}h_{j+c-i-1}h_ih_{i+1}}{h_jh_{j+1}h_{c-2}h_{c-1}} +
\\ & -t^{2(c-1)} \frac{h_1h_{i-c+1}h_{j-i-1}}{ h_{j+1}h_{c-2} }  -t^{4c}  \frac{h_{i-c}h_{i-c+1}h_{j-i-2}h_{j-i-1}}{h_{c-2}h_{c-1}h_jh_{j+1}},
\end{align*}
\noindent formula \eqref{caso k-i=2} holds if and only if $F(i,j,c)=1$. Observe that $F(i,j,c)$ does make sense for every $c\geq 2$ and for every positive integers $i,j$.
In the following, we repeatedly apply the equality
$t^{2\alpha}h_\beta = h_{\alpha+\beta} - h_{\alpha-1}$, which holds for every $\alpha, \beta \geq 0$.

\medskip
We now show that $F(i,j,3)=1$ for $j\geq i$, where 
\begin{align*}
F(i,j,3)=\frac{h_ih_{i+1}}{h_jh_{j+1}}\frac{h_{j-i+1}h_{j-i+2}}{P_3}  -t^4 \frac{h_{i-2}h_{j-i-1}}{h_{j+1}} -t^{12}\frac{h_{j-i-2}h_{j-i-1}h_{i-2}h_{i-3}}{P_3 h_j h_{j+1}}.
\end{align*} 
Although we should consider $i\geq c\geq 3$ due to  \eqref{eq:geom-cond}, symbolically we easily obtain 
$$F(2,j,3)=\frac{h_2h_3}{h_jh_{j+1}} \frac{h_jh_{j-1}}{h_1h_2} -t^4\frac{h_{j-3}h_1}{h_{j+1}h_1}= \frac{h_3}{h_{j+1}} \frac{h_{j-1}}{h_1} - \frac{(h_{j-1}-h_1) h_1}{h_1h_{j+1}} =$$
$$\frac{(h_3-h_1)h_{j-1}+h_1^2}{h_1h_{j+1}} =\frac{t^4(1+t^2)h_{j-1}+h_1^2}{h_1h_{j+1}}= \frac{h_{j+1}-h_1+h_1}{h_{j+1}}=1.$$
Hence, by induction, we can assume $F(i-1,j,3)=1$ and prove that $F(i,j,3)-F(i-1,j,3)$ is null, for every $i>2$. We compute
\begin{align*}
F(i,j,3)-F(i-1,j,3)=
\frac{ h_i h_{j-i+2}(h_{i+1}h_{j-i+1} - h_{i-1}h_{j-i+3}) } {h_jh_{j+1}P_3} + \\
 - \frac{ t^4 h_j P_3 (h_{i-2}h_{j-i-1}-h_{i-3}h_{j-i}) }{h_jh_{j+1}P_3}- \frac{ t^{12} h_{j-i-1} h_{i-3} (h_{j-i-2}h_{i-2}   -  h_{j-i}h_{i-4})}{h_j h_{j+1}P_3}.
\end{align*}

\noindent So, letting 
\vskip 2mm
$G_0:= h_i h_{j-i+2} (h_{i+1}h_{j-i+1} - h_{i-1}h_{j-i+3})$

$G_1:=  -t^4 h_j P_3 (h_{i-2}h_{j-i-1}  - h_{i-3}h_{j-i})$

$G_2:=  -t^{12} h_{j-i-1} h_{i-3} (h_{j-i-2}h_{i-2}   -  h_{j-i}h_{i-4})$,

\noindent we obtain $F(i,j,3)-F(i-1,j,3)=0$ if and only if $G_0+G_1+G_2=0$. 

\medskip
\noindent Being $h_{i+1}=t^4h_{i-1}+h_1$ e $h_{j-i+3}=t^4h_{j-i+1}+h_1$, $G_0$ becomes:
\begin{itemize}
\item $G_0=h_ih_{j-i+2}h_1 (h_{j-i+1} - h_{i-1})$.
\end{itemize}
Being $h_{i-2}= t^2h_{i-3}+h_0$, $h_{j-i}=t^2h_{j-i-1}+h_0$ and then $t^4h_{i-3}=h_{i-1}-h_1$, $t^4h_{j-i-1}= h_{j-i+1}-h_1$, $G_1$ becomes:
\begin{itemize}
\item $G_1= h_jP_3(h_{i-1}-h_{j-i+1})$.
\end{itemize}
Being $h_{j-i}= t^4h_{j-i-2}+h_1$ e $h_{i-2}= t^4 h_{i-4}+h_1$, $G_2$ becomes
\begin{itemize}
\item $G_2=  t^6 h_{j-i-1} h_{i-3}h_1  (h_{i-1}-h_{j-i+1})$.
\end{itemize}
So, we obtain 
$G_0+G_1+G_2= h_1   (h_{i-1}-h_{j-i+1}) (h_i h_{j-i+2}- h_jh_2 -t^6 h_{j-i+1}h_{i-3})=0.$

\medskip
Recall that we are assuming $k-i\leq k-c$, hence $i\geq c$, and we know that $F(i,j,3)=1$. So, our thesis now is $F(i,j,c)-F(i,j,c-1)=0$. Assuming $c> 3$, we compute
$$F(i,j,c)-F(i,j,c-1)= 
 \frac{h_{j+c-i-2}h_{j+c-i-1}h_ih_{i+1}}{h_jh_{j+1}h_{c-2}h_{c-1}} -  \frac{h_{j+c-i-3}h_{j+c-i-2}h_ih_{i+1}}{h_jh_{j+1}h_{c-3}h_{c-2}} + $$ 
$$  -t^{2(c-1)} \frac{h_1h_{i-c+1}h_{j-i-1}}{ h_{j+1}h_{c-2} } +  t^{2(c-2)} \frac{h_1h_{i-c+2}h_{j-i-1}}{ h_{j+1}h_{c-3} } + $$
$$  -t^{4c}  \frac{h_{i-c}h_{i-c+1}h_{j-i-2}h_{j-i-1}}{h_{c-2}h_{c-1}h_jh_{j+1}}  +  t^{4(c-1)}  \frac{h_{i-c+1}h_{i-c+2}h_{j-i-2}h_{j-i-1}}{h_{c-3}h_{c-2}h_jh_{j+1}}.$$
So, letting
\vskip 2mm
$Q_0 :=  h_{j+c-i-2}h_{j+c-i-1}h_ih_{i+1}h_{c-3} -  h_{j+c-i-3}h_{j+c-i-2}h_ih_{i+1}h_{c-1}$

$Q_1 :=   -t^{2(c-1)} h_1h_{i-c+1}h_{j-i-1}h_jh_{c-1}h_{c-3} +  t^{2(c-2)} h_1h_{i-c+2}h_{j-i-1}h_jh_{c-1}h_{c-2} $

$Q_2 :=  -t^{4c}  h_{i-c}h_{i-c+1}h_{j-i-2}h_{j-i-1}h_{c-3} +  t^{4(c-1)}  h_{i-c+1}h_{i-c+2}h_{j-i-2}h_{j-i-1}h_{c-1} $,
\vskip 2mm
\noindent we have $F(i,j,c)-F(i,j,c-1)=0$ if and only if $Q_0+Q_1+Q_2=0$.

\medskip
\noindent Being $h_{j+c-i-1}= t^4 h_{j+c-i-3}+h_1$ and $t^4h_{c-3}=h_{c-1}-h_1$, then $Q_0$ becomes
\begin{itemize}
\item $Q_0=h_{j+c-i-2}h_1h_ih_{i+1}(h_{c-3}- h_{j+c-i-3})= -t^{2(c-2)}h_{j+c-i-2}h_1h_ih_{i+1} h_{j-i-1}$.
\end{itemize}
Being $t^{2(c-1)}h_{i-c+1}=h_{i}-h_{c-2}$  and  $t^{2(c-2)}h_{i-c+2}=h_{i}-h_{c-3}$, then $Q_1$ becomes
\begin{itemize}
\item $Q_1 = t^{2(c-2)}h_1h_ih_jh_{c-1}h_{j-i-1}$.
\end{itemize}
Being $t^4h_{i-c}=h_{i-c+2}-h_1$, then $Q_2$ becomes
\begin{itemize}
\item $Q_2= t^{4(c-1)} h_{i-c+1}h_{j-i-2}h_{j-i-1}((h_{i-c+2}-h_1)h_{c-3} +   h_{i-c+2}h_{c-1})$.
\end{itemize}
Note that, if $j=i$ or $j=i+1$, then $Q_0+Q_1=0=Q_2$.
So, we now consider the less obvious case $j>i+1$. Observe that
$$Q_1+Q_0= t^{2(c-2)}h_1h_ih_{j-i-1}(h_jh_{c-1}-h_{i+1}h_{j+c-i-2}) = t^{2(c-2)}h_1h_ih_{j-i-1} h_{j-i-2}(h_{c-1}-h_{i+1})=$$
$$= -t^{2(c-2)} h_1h_ih_{j-i-1} h_{j-i-2} t^{2c} h_{i-c+1}.$$
We are done, because:
$$Q_1+Q_0 +Q_2= t^{4(c-1)}h_{i-c+1}h_{j-i-2}h_{j-i-1}(-h_1h_i+(h_{i-c+2}-h_1)h_{c-3} + h_{i-c+2}h_{c-1}) =$$
$$= t^{4(c-1)}h_{i-c+1}h_{j-i-2}h_{j-i-1} h_{i-c+2}(t^{2(c-1)}+t^{2(c-2)} -h_1t^{2(c-2)})=0.$$

\subsection{Case $2=k-c < k-i$}
In this case we have $r+i=k=c+2$, $\ell=j+c$, $c=r+i-2$, and the global polynomial identity becomes:
\begin{equation}\label{caso k-c=2}
\begin{array}{c}
\displaystyle{
\frac{P_j P_{j+r-2}}{P_i P_{j-i} P_r P_{j-2}} = \frac{P_{r+i-2}}{P_r P_{i-2}} \frac{P_{r+j}}{P_{r+i}P_{j-i}}
} + \\
\displaystyle{ 
+ t^{2(i-1)} \frac{P_2 P_{r+i-2}P_{r+j-1}}{P_{r-1}P_{i-1}P_{r+i}P_{j-i-1}} + t^{4i} \frac{P_2 P_{r+i-2} P_{r+j-2}}{P_2P_{r-2}P_i P_{r+i} P_{j-i-2}}
}
\end{array}
\end{equation}
where only the parameters $i,j,r$ appear. Note that formula \eqref{caso k-c=2} does make sense for every $2\leq r$ and $2\leq i \leq j-2$. Let
\vskip 2mm
$\displaystyle{K:=\frac{P_j P_{j+r-2}}{P_i P_{j-i} P_r P_{j-2}}}$, $\displaystyle{E:= \frac{P_{r+i-2}}{P_r P_{i-2}} \frac{P_{r+j}}{P_{r+i}P_{j-i}}}$,

$\displaystyle{E_1:= t^{2(i-1)} \frac{P_2 P_{r+i-2}P_{r+j-1}}{P_{r-1}P_{i-1}P_{r+i}P_{j-i-1}}}$,  $\displaystyle{E_2:=t^{4i} \frac{P_2 P_{r+i-2} P_{r+j-2}}{P_2P_{r-2}P_i P_{r+i} P_{j-i-2}}}$.
\vskip 2mm
Analogously to the previous case, we compute
\vskip 2mm
$\displaystyle{K=E\left[\frac{h_{j-1}h_{j-2}h_{r+i-1}h_{r+i-2}}{h_{i-1}h_{i-2}h_{r+j-1}h_{r+j-2}} \right]}$, 
$\displaystyle{E_1=E \left[ t^{2(i-1)} \frac{h_{r-1}h_1h_{j-i-1}h_{i-1}h_{r+j-2}} {h_{r+j-1}h_{i-2}h_{i-1}h_{r+j-2}}\right]}$, 
\vskip 2mm
$\displaystyle{E_2=E\left[t^{4i}\frac{h_{r-2}h_{r-1}h_{j-i-2}h_{j-i-1}} {h_{r+j-2}h_{r+j-1}h_{i-2}h_{i-1}} \right]}$.

\noindent Hence, letting 
\begin{align*}
F\!F(i,j,r)&:= \frac{h_{j-1}h_{j-2}h_{r+i-1}h_{r+i-2}}{h_{i-1}h_{i-2}h_{r+j-1}h_{r+j-2}} + \\
& - t^{2(i-1)} \frac{h_{r-1}h_1h_{j-i-1}h_{i-1}h_{r+j-2}} {h_{r+j-1}h_{i-2}h_{i-1}h_{r+j-2}} - t^{4i}\frac{h_{r-2}h_{r-1}h_{j-i-2}h_{j-i-1}} {h_{r+j-2}h_{r+j-1}h_{i-2}h_{i-1}},
\end{align*}
\noindent formula \eqref{caso k-c=2} holds if and only if $F\!F(i,j,r)=1$. Observe that $F\!F(i,j,r)$ does make sense for all integers $j\geq i\geq 2$ and $r\geq 0$. We easily check that $F\!F(i,j,0)=1$. So, we now assume $r>0$ and prove that $F\!F(i,j,r)-F\! F(i,j,r-1)=0$. Arguing as in the case $k-i=2$, we let

$H_0:= h_{r+j-3}h_{j-1}h_{j-2}h_{r+i-1}h_{r+i-2} - h_{r+j-1}h_{j-1}h_{j-2}h_{r+i-2}h_{r+i-3}$

$H_1:= t^{2(i-1)} h_{r+j-1}h_{r-2}h_1h_{j-i-1}h_{i-1}h_{r+j-3} - t^{2(i-1)}h_{r+j-3} h_{r-1}h_1h_{j-i-1}h_{i-1}h_{r+j-2}$

$H_2:= t^{4i} h_{r+j-1}h_{r-3}h_{r-2}h_{j-i-2}h_{j-i-1} - t^{4i} h_{r+j-3} h_{r-2}h_{r-1}h_{j-i-2}h_{j-i-1}$

\noindent so that the thesis now is $H_0+H_1+H_2=0$.

\medskip
We apply in $H_0$ the following replacements, in the given order:
first $h_{r+i-1} = t^4h_{r+i-3}+h_1$ and $h_{r+j-1}=t^4h_{r+j-3}+h_1$, then $h_{r+j-2}-h_{r+i-2}=t^{2(r+i-2)}h_{j-i-1}$. Hence, $H_0$ becomes
\begin{itemize}
\item $H_0=h_{j-1}h_{j-2}h_{r+i-2}h_1 t^{2(r+i-2)}h_{j-i-1}$.
\end{itemize}
We apply in $H_1$ the following replacements, in the given order: first $h_{r+j-1}=t^2h_{r+j-2}+h_0$ and $h_{r-1}=t^2h_{r-2}+h_0$, then $h_{r-2}-h_{r+j-2}= -t^{2(r-1)}h_{j-1}$.
Hence, $H_1$ becomes
\begin{itemize}
\item $H_1=t^{2(i-1)}h_{r+j-3}h_{j-i-1}h_{i-1}(-t^{2(r-1)}h_{j-1})h_1$.
\end{itemize}
We apply in $H_2$ the following replacements, in the given order:  first $h_{r+j-1}=t^4h_{r+j-3}+h_1$ and $h_{r-1}=t^4h_{r-3}+h_1$, then $h_{r-3}-h_{r+j-3} = -t^{2(r-2)}h_{j-1}$.  Hence, $H_2$ becomes
\begin{itemize}
\item $H_2= - t^{2i}t^{2(r+i-2)} h_{r-2}h_{j-i-2}h_{j-i-1}h_1 h_{j-1}$.
\end{itemize}
\noindent Thus, we now have 
$$H_0+H_1+H_2=  t^{2(r+i-2)} h_1h_{j-1}h_{j-i-1}\Bigl(h_{j-2}h_{r+i-2} -  h_{r+j-3}h_{i-1} - t^{2i} h_{r-2}h_{j-i-2}\Bigr).$$
Being $h_{j-2}= t^{2(j-i-1)}h_{i-1}+h_{j-i-2}$ and $h_{r+j-3}=t^{2(j-i-1)}h_{r+i-2}+h_{j-i-1}$, we finally obtain
$$H_0+H_1+H_2=  t^{2(r+i-2)} h_1h_{j-1}h_{j-i-1}h_{j-i-2}\Bigl(h_{r+i-2} - h_{i-1} - t^{2i} h_{r-2}\Bigr)=0.$$
\end{small}

\end{document}